\documentclass[12pt]{amsart}
\usepackage[all]{xy}

\usepackage{graphicx}
\usepackage{amsthm}
\usepackage{amstext}
\usepackage{amsmath,amscd,amsfonts}
\usepackage{amssymb}
\usepackage{mathrsfs}
\usepackage{mathtools}
\usepackage[all]{xy}
\usepackage{color}
\usepackage[table]{xcolor}
\usepackage[normalem]{ulem}

\newcommand{\A}{\mbox{${\mathcal A}$}}
\newcommand{\B}{\mbox{${\mathcal B}$}}

\newcommand{\E}{\mbox{${\mathcal E}$}}

\newcommand{\J}{\mbox{${\mathcal J}$}}

\renewcommand{\S}{\mbox{${\mathcal S}$}}
\newcommand{\T}{\mbox{${\mathcal T}$}}

\newcommand{\sspan}{\operatorname{span}}

\newcommand{\complexos}{\mathbb{ C}}
\newcommand{\naturais}{\mathbb {N}}

\newtheorem{theorem}{Theorem}[section]
\newtheorem{lemma}[theorem]{Lemma}
\newtheorem{corollary}[theorem]{Corollary}
\newtheorem{proposition}[theorem]{Proposition}

\theoremstyle{remark}
\newtheorem{remark}[theorem]{\bf Remark}
\theoremstyle{definition}
\newtheorem{example}[theorem]{\bf Example}
\newtheorem{conjecture}[theorem]{Conjecture}
\numberwithin{equation}{section}

\def\sqr#1#2{{\,\vcenter{\vbox{\hrule height.#2pt\hbox{\vrule width.#2pt
height#1pt \kern#1pt\vrule width.#2pt}\hrule height.#2pt}}\,}}

\def\J{\mathcal{J}}
\def\E{\mathcal{E}}
\def\span{\text{span}}
\def\supp{\Phi_{\mathcal J}}
\def\suppe{\Phi_{\mathcal J}^e}
\def\phie{\Phi^e}
\def\supm{\Phi_{{\mathcal M}(\Phi,\Psi)}}
\def\supme{\Phi_{{\mathcal M}(\Phi,\Psi)}^e}

\makeatletter
\@namedef{subjclassname@2020}{%
  \textup{2020} Mathematics Subject Classification}
\makeatother
\begin{document}

\title[Bimodules of Banach space nest algebras]{Bimodules of Banach space nest algebras}

\author[L. Duarte]{Lu\'is Duarte}
\address{Dipartimento di Matematica \\
Universit\`a degli Studi di Genova \\
Via Dodecaneso 35\\
I-16146 Genova, Italia}
\email{duarte@dima.unige.it}

\author[L. Oliveira]{Lina Oliveira}
\address{Center for Mathematical Analysis\\
Geometry and Dynamical Systems\\
{\sl{and}}
Department of Mathematics\\
Instituto Superior T\'ecnico \\
Universidade de Lisboa\\
Av. Rovisco Pais\\
1049-001 Lisboa, Portugal}
\email{lina.oliveira@tecnico.ulisboa.pt}

\date{\today}
\thanks{Partially supported by FCT/Portugal grant  UID/MAT/04459/2013 and BL186/2016}
\keywords{Nest algebra, bimodule,  essential support function support function, support function pair.}

\begin{abstract}
We extend to Banach space nest algebras the theory of essential supports and support function pairs of their bimodules, thereby  obtaining Banach space counterparts of long established results for Hilbert space nest algebras. Namely, given a Banach space nest algebra $\A$, we charaterise
 the maximal and the minimal $\A$-bimodules having a given essential support function or  support function pair.  These characterisations are complete except for the minimal $\A$-bimodule corresponding to a  support function pair, in which case we make some headway.  
 We also show   that the weakly closed bimodules of a Banach space nest algebra are exactly those that are  reflexive operator spaces. To this end, we crucially prove that reflexive bimodules determine uniquely a certain class of admissible support functions.
\end{abstract}

\subjclass[2020]{47A35, 47L15, 46H10, 46H25, 16D20}

\maketitle

\section{Introduction}
Nest algebras on Hilbert space were introduced in 1963 by J. R. Ringrose as a generalisation of a class of operator algebras including, for example, the algebra of upper triangular matrices and some of the maximal triangular algebras appearing in the work
of Kadison and Singer (see \cite{kadisonsinger},\cite{Ringrose}).
 These algebras have attracted a vast amount of research over the years, with Davidson's book \cite{Davidsonbook} being the standard reference for a comprehensive account of their basic theory. 

A natural question is, of course, to what extent results in the Hilbert space setting can
be seen to hold for Banach spaces, a question already hinted at in \cite{Ringrose}. Although many authors have tackled this problem, it is
still  the case that nest algebras on Banach spaces are far less explored in the literature.
One of the possible reasons for this being the difficulty presented by the absence of orthogonal
projections intrinsically associated with the structure of an abstract Banach space, a key tool used in the investigation of abstract Hilbert space nest algebras.

A concrete example of this situation is provided by the bimodules of a  nest
algebra. In fact, (weakly closed) bimodules of a given nest algebra on a Hilbert space
were completely characterized by Erdos and Power as early as 1982 (see \cite{ErdosPower}). Those
were characterized in terms of order homomorphisms on the corresponding nest, later
to be called support functions (see \cite{DDH}). Similar
descriptions have been shown to hold in the case of reflexive operator algebras and, more
generally, of reflexive spaces of operators (see \cite{BO}, \cite{han}).
 Davidson, Donsig and Hudson \cite{DDH} went further to introduce the essential support
function of a given norm closed bimodule, which allowed for finding the maximal and
the minimal norm closed bimodule having the same essential support function. It is also
investigated in \cite{DDH} how, conversely, essential support functions and the so-called support
function pairs determine themselves bimodules.

Although several results have already been generalized from the Hilbert space setting 
to that of Banach spaces (see, e.g., \cite{han2}, \cite{LiLi}, \cite{LiLu}, \cite{spadounakis}),  that has not been the case in what concerns essential supports and their
associated bimodules.

In this work, the central subjects under scrutiny  are bimodules of nest algebras on
Banach spaces with a particular focus on support functions (Section \ref{s_bim}), essential
support functions, support function pairs  (Section \ref{s_essbim}) and their associated bimodules. The approach to the
investigation of bimodules of nest algebras on Banach spaces through support functions
appears in the literature, although not quite in the way it is presented in Section \ref{s_bim}.
 By contrast,  Section  \ref{s_essbim} consists entirely of new material, to the best of our knowledge. In this section,
 we introduce the  definitions   of essential support function and support function pair in the Banach space context which will allow for obtaining results analogous to some of those in \cite{DDH}.
 It is however worth pointing out that the extension of existing results for Hilbert space nest algebras to the Banach space setting is neither automatic nor without difficulties. In fact, for example, the similarity results for nests in Hilbert spaces have been shown to fail for Banach spaces, e.g., the space $L^1$ (see \cite{ALW}). The absence of an intrinsic notion of orthogonality has also proved to be crucial as an impediment to an easy generalisation of results.

This work  is organised as follows. 
In Section \ref{s_bim}, bimodules and support functions are introduced and their interplay is
investigated, an investigation that does not seem to be available in the
existing literature. The characterisation of reflexive bimodules in terms of support
functions leads to proving  that reflexive bimodules are in 1-1 correspondence with the admissible   support functions fixing the $\{0\}$ set. 

The  main result of this section (Theorem \ref{closedcar})  shows that the weakly closed bimodules of a Banach space nest algebra are exactly those that are  reflexive operator spaces. 
 
 Finite rank operators play an important  role throughout Section 2. Crucially, given a support function, the finite rank operators in the associated bimodule \eqref{01} are   shown to be decomposable within said bimodule (Theorem \ref{finiteranksum}). In fact, this bimodule is precisely the  maximal (weakly closed) bimodule   having this support function, should the latter be an admissible support function fixing $\{0\}$ (see Remark \ref{rem3}).

 In Section \ref{s_essbim}, we introduce the notions of essential support function and support function pair on a given
nest and their associated bimodules of the corresponding nest algebra, obtaining results
analogous to those in \cite{DDH}, albeit our definition of essential support function be  an extension rather than  a 
generalisation of that given in \cite{DDH}. 
To mention one of the differences, notwithstanding   each essential support function still determining a largest $\T(\E)$-bimodule, the latter is not necessarily closed, unlike its Hilbert space counterpart  (see Example \ref{exa}). 

It is also the case that we are able to show that each essential support function $\Psi$  determines a smallest norm closed $\T(\E)$-bimodule whose essential support is $\Psi$ (see Theorem \ref{p_bimodess}).
 More precisely, we obtain  the  maximal  and the minimal bimodules with a given
essential support function (Theorem \ref{p_bimodess}) and the maximal bimodule having some fixed support function pair (Theorem \ref{t_maxpair}). As to finding the minimal bimodule corresponding to a given support function pair, we make some headway inasmuch as we are able to construct a bimodule which we believe to be minimal amongst the norm closed $\T(\E)$-bimodules having this support function pair (see Proposition \ref{p_m^0}, Corollary \ref{c_m^0} and Conjecture \ref{co_final}).

Some of the results mentioned above require the nest to be restricted to a class. However, this restriction is either similar to that existing  in the Hilbert space setting   for the corresponding result  or covers a vast class of nest algebras.\\

We end this section with  some background and a few facts needed in the sequel.

 Let $\mathcal B(X)$ be the algebra of the bounded linear operators on a complex Banach space $X$ and let $X^*$ the space consisting  of the bounded linear functionals on $X$. 
 Given a subset $S$ of $X$, let $[S]$ be the smallest norm closed subspace of $X$ containing $S$. 
The \emph{annihilator} $S^{\perp}$ of $S$ is the weak*-closed subspace of $X^*$ defined by 
$$S^{\perp}=\{f\in X^*: f(S)=\{0\}\}$$
and the \emph{pre-annihilator} $M_{\perp}$ of a subset $M$ of $X^*$ is the norm closed subspace of $X$ defined by 
$$M_{\perp}=\{x\in X: f(x)=0\quad\forall f\in M\}.$$
It is well-known that $(S^{\perp})_\perp=[S]$ and that $(M_\perp)^\perp=\overline M^{w^*}$, where $\overline M^{w^*}$ is the closure of $M$ in the weak*-topology.

Given a subset $A$ of $\mathcal B(X)$ we denote by $\overline A^{\mbox{{\tiny SOT}}}$ and $\overline A^{\mbox{{\tiny WOT}}}$ the closure of $A$ in the strong and weak operator topologies, respectively, and    $\overline A$ will be used when  referring to the closure of $A$ in the norm topology.

The set consisting of all the closed subspaces of $X$ is partially ordered by inclusion. Consider the usual operations of meet $\wedge$ and join $\vee$ on this set: for a  family  $\{M_{i}\colon i\in I\}$ of  closed subspaces of $X$, its \emph{infimum} $\wedge_{i\in I} M_{i}$ coincides with the intersection $\cap_{i\in I} M_{i}$ and its \emph{supremum} 
$\vee_{i\in I} M_{i}$ is  the smallest closed subspace containing all the elements in  $\{M_{i}\colon i\in I\}$, i.e., 
$$\wedge_{i\in I} M_{i}=\cap_{i\in I}M_i,  \qquad \quad \vee_{i\in I}M_i=\overline{\span}\left(\cup_{i\in I}M_i\right ).$$

A family  of closed subspaces of $X$ containing $\{0\}$ and $X$ and closed under meets $\wedge$ and joins $\vee$  is called a \emph{subspace lattice}.  A \emph{nest} $\E$ is a totally ordered subspace lattice.
The \emph{nest algebra} $\mathcal T(\mathcal E)$ associated with $\mathcal E$ consists of the operators in  $\mathcal B(X)$ leaving invariant  every subspace of $\mathcal E$, i.e.,
$$\mathcal T(\mathcal E)=\{T\in\mathcal B(X):TE\subseteq E \quad\forall E\in \mathcal E\}.$$
The algebra $\T(\mathcal E)$ is {weakly (operator) closed and reflexive}, i.e., {$\operatorname{Alg}\operatorname{Lat}  \T(\E)=\T(\E)$.} In fact, the nest $\E$ is itself a reflexive subspace lattice, that is,  
$\operatorname{Lat}\operatorname{Alg} (\E)=\E$. 
Recall that, given a subspace  lattice $\mathcal{L}$, $\operatorname{Alg}(\mathcal{L})$ is the set of operators in $\B(X)$ that leave invariant each element of $\mathcal{L}$, and that given a subset $\S$ of $\B(X)$, $\operatorname{Lat} \S$ is the set of closed subspaces invariant under $\S$ 
(see \cite{A, Halmos, Halmos1, Ringrose, spadounakis}). 

\begin{lemma}\label{finitesubspacenotzero}
Let $\mathcal E$ be a nest in a Banach space $X$, let $W\neq\{0\}$ be a finite dimensional subspace of $X$ and let 
$L=\wedge\{N\in\E:N\cap W\neq\{0\}\}$. Then $L$ is the smallest element in $\E$ whose intersection with $W$ strictly contains $\{0\}$. 
\end{lemma}
\begin{proof} 
It suffices to show that $L\cap W$ does not coincide with $\{0\}$. Notice that 
$$
L\cap W=\wedge\{N\cap W:N\in\E, N\cap W\neq\{0\}\}
$$
and that $\{N\in\E:N\cap W\neq\{0\}\}$ is non-empty, since it contains $X$.

Since  $\E$ is a chain and $W$ is finite dimensional,  for all $N,N'\in\E$, we have that $N\cap W=N'\cap W$ if and only if $\dim N\cap W=\dim N'\cap W$. Hence,  there exist finitely many  intersection sets $N\cap W$ when   $N\in \E$. Consequently, $L\cap W$ is the intersection of a finite number of totally ordered non-zero subspaces, yielding  that $L\cap W$ is non-zero.   
\end{proof}

In what follows we adopt the following notation: the symbol $\subset$ will be used exclusively for strict inclusions whereas  $\subseteq$ encompasses also the equality between sets.

For each   $E$ in a nest $\E$,  define 
$$
E_-=\vee\{F\in\E:F\subset E\}\quad\text{  and  }\quad E_+=\wedge\{F\in\E:E\subset F\}.
$$ 

\begin{lemma}\label{perpspan}
Let $E$ be in a nest $\E$. Then the closure of $\text{span}\{N^{\perp}:N\in\E,N_+\supset E\}$ in the weak*- topology coincides with $E^{\perp}$.
\end{lemma}
\begin{proof}
Let  
$\overline{\text{span}}^{w^*}\{N^{\perp}:N_+\supset E\}
$ be the closure in the weak*-topology of the subspace $\text{span}\{N^{\perp}:N\in\E,N_+\supset E\}$. Then 
$$\begin{aligned}
\overline{\text{span}}^{w^*}\{N^{\perp}:N_+\supset E\}=\left(\left(\cup\{N^{\perp}:N_+\supset E\}\right)_{\perp}\right)^{\perp}
&=\left(\cap\{(N^{\perp})_{\perp}:N_+\supset E\}\right)^{\perp}\\
&=\left(\cap\{N:N_+\supset E\}\right)^{\perp}=E^{\perp},
\end{aligned}
$$
as required.
\end{proof}

A nest $\E$ is said to be \emph{continuous} if $E_-=E$, for all $E\in\E$.
Observe that, if $E_-\subset E$, then  $E=(E_-)_+$, and that, if $E\subset E_+$, then $E=(E_+)_-$. It is easily seen that 
\begin{equation}\label{01}
E=\vee\{N_+:N\subset E\}=\vee\{N:N_-\subset E\}=\wedge\{N_-:E\subset N\}=\wedge\{N:E\subset N_+\}.
\end{equation}

A rank-$1$ operator $f\otimes w$ on $X$ is defined, by means of a bounded linear functional $f\in X^*$ and a vector $w\in X$, as  $f\otimes w(x)=f(x)w$.

 The next lemma is essentially  \cite[Lemma 3.3]{Ringrose}.

\begin{lemma}\label{rankonenestalg}
Let $\mathcal E$ be a  nest in a complex Banach space $X$ and let  $f\otimes w$ be a rank-$1$ operator in $\B(X)$. The following assertions are equivalent. 
\begin{itemize}
\item[(i)] The rank-$1$ operator $f\otimes w$ lies in $\mathcal T(\mathcal E)$.

\item[(ii)] There exists $E\in\mathcal E$ such that $w\in E$ and $f\in (E_-)^{\perp}$.

\item[(iii)] There exists $E\in\mathcal E$ such that $w\in E_+$ and $f\in E^{\perp}$.
\end{itemize}
\end{lemma}

Finite rank operators  play a decisive role in the theory of Hilbert space nest algebras and  have been thoroughly investigated with respect to their density and decomposability (e.g.,  \cite{E, LL, MO, O3, O2}). In fact, they have shown to be  crucial in the characterisation of associative, Jordan and Lie  modules of nest algebras on Hilbert space (e.g., \cite{LiLi, LiLu, O2, OS}).  Still, more recently, they  have  been given a prominent place in the  related context of triangularizability (e.g., \cite{Drn}).

 Finite rank operators have also a relevant place both in the theory  of Banach space nest algebras and  in  the present work.  Similarly to what happens in the Hilbert space setting, each finite rank operator in a Banach space nest algebra $\T(\E)$ is decomposable into a finite sum of rank-$1$ operators in $\T(\E)$ (\cite[Theorem 3.1]{LiLu}). Moreover,  the rank-$1$ operators are strongly dense in $\T(\E)$.

\begin{theorem}[\cite{spadounakis},Theorem 3]\label{rankstrong}
Let $\E$ be a nest of subspaces of a Banach space $X$ and let $\mathcal T(\E)_0$ be the set of finite rank operators in the nest algebra $\T(\E)$. Then the closure of $\mathcal T(\E)_0$ in  the strong operator topology coincides with $\mathcal T(\E)$.
\end{theorem}



%

\section{Support functions}\label{s_bim}

A  subspace $\J$ of $\mathcal B(X)$ is said to be a \emph{$\mathcal T(\mathcal E)$-bimodule} if $\J\mathcal T(\mathcal E),\mathcal T(\mathcal E)\J\subseteq \J$. In what follows, for simplicity, $\mathcal T(\mathcal E)$-bimodules may be referred to as bimodules. Following \cite{DDH}, we shall establish a correspondence between bimodules and support functions.
 A \emph{support function} on a nest $\E$ is an order preserving map $\Phi:\E\to\E$, i.e., a map such that  $\Phi(N)\subseteq \Phi(M)$ whenever  $N\subseteq M$. A support function $\Phi$ is called an \emph{admissible support function} if in addition $\Phi$ is \emph{left continuous}, that is, for every $N\in\mathcal E\backslash \{\{0\}\}$ we have
  $$
  \vee_{E\subset N}\Phi(E)=\Phi(N_-).
  $$

Support functions on $\E$ form a partially ordered set.   Let $\Phi_1$ and $\Phi_2$ be support functions. We write $\Phi_1\le\Phi_2$ if, for all $E\in\E$,  $\Phi_1(E)\subseteq\Phi_2(E)$. 

Given a support function $\Phi$ on $\E$, we define $\Phi_-$ to be the greatest admissible support function $\Phi_-$ such that $\Phi_-\le\Phi$. It can be seen that $\Phi_-$ is given by $\Phi_-(\{0\})=\Phi(\{0\})$ and, for $E\neq\{0\}$, 
$$
\displaystyle{\Phi_-(E)=\vee_{F_-\subset E}\Phi(F)=\left\{
\begin{array}{ll}
\Phi(E) & \text{if }E_-\subset E\\
\displaystyle{\vee_{F\subset E}\, \Phi(F)} & \text{if }E_-= E
\end{array}\right.}
$$

Given a $\T(\E)$-bimodule $\mathcal J$, define the function $\Phi_{\mathcal J}$ on $\E$ by
$$
\begin{aligned}
\Phi_{\mathcal J}&:\mathcal E\longrightarrow\mathcal E \\ 
&E\mapsto [\mathcal JE].
\end{aligned}
$$

\begin{proposition}\label{p_phiJ}
Let $\E$ be a nest, let $\T(\E)$ be the corresponding nest algebra and let $\J$ be a 
$\T(\E)$-bimodule. Then $\supp$ is an admissible support function on $\E$.
\end{proposition}

\begin{proof}
Observe that $\supp$ is well-defined. In fact, since $\J$ is a bimodule,  we have, for all $E\in\E$, 
$$
\mathcal T(\E)[\J E]=[\mathcal T(\E)\J E]\subseteq [\J E].
$$
  By the reflexivity of  $\E$,  it  follows that $[\J E]\in\E$. 
 
 Clearly, $\supp$ is order preserving. Hence,  $\supp$ is a support function on $\E$.
It remains to show  that $\Phi_\J$ is left continuous. 

Let $N\in\E\backslash\{\{0\}\}$. 
We only show that 
$
\vee_{E\subset N}\, \Phi_\J(E)\supseteq\Phi_\J(N_-),
$ 
since the reverse inclusion is obvious. 
 Let  $f$ be a functional in  $X^*$ and  let $T$ be an operator in $\mathcal J$. By the continuity of $f$,    if $f(TE)=\{0\}$ for every $E\subset N$, then $f(TN_-)=\{0\}$. 
 It follows that, if  $f(\mathcal JE)=\{0\}$ for every $E\subset N$, then   $f(\mathcal JN_-)=\{0\}$.  Hence, 
$$
\left(\vee_{E\subset N}\, \Phi_\J(E)\right)^{\perp}\subseteq(\Phi_\J(N_-))^{\perp},
$$
 from which follows that 
 $$
\vee_{E\subset N}\, \Phi_\J(E)=(\left(\vee_{E\subset N}\, \Phi_\J(E)\right)^{\perp})_\perp\supseteq((\Phi_\J(N_-))^{\perp})_\perp=\Phi_\J(N_-),
$$
 as required.
\end{proof}

Given a support function $\Phi$ on a nest $\mathcal E$,  let $\mathcal M(\Phi)$ be the  $\T(\E)$-bimodule
 defined by
\begin{equation}\label{01}
\mathcal M(\Phi)=\{T\in\mathcal B(X):TE\subseteq \Phi(E)\quad\forall E\in\mathcal E\}.
\end{equation}

Recall that a subspace $A$ of $\mathcal B(X)$ is \emph{reflexive} if it coincides with its \emph{reflexive cover}
$$
\text{Ref }A=\{T\in\mathcal B(X):Tx\in[Ax], \, \forall x\in X\}.
$$

The next proposition characterises the reflexive bimodules in terms of support functions. For a proof, see, for example,  [\cite{LiLi}, Proposition 2.6].

\begin{proposition}\label{p_reflexbim} Let $\E$ be a nest, let $\T(\E)$ be the corresponding nest algebra and let $\J$ be a 
$\T(\E)$-bimodule.  Then $\J$ is reflexive if and only if $\mathcal J=\mathcal M(\Phi_{\mathcal J})$. 
\end{proposition}

It follows from this proposition that the mapping $\Phi\mapsto \mathcal M(\Phi)$, from the set of admissible
 support functions to the set of reflexive bimodules, is surjective. However, this mapping is not injective as, for example, changing the image of $\{0\}$ does not change the associated bimodule. Nevertheless, assuming left continuity, we can prove injectivity up to the image of $\{0\}$. To the best of our knowledge, this is new even in the Hilbert space setting, where the counterpart of the next proposition has been proved for continuous nests only (cf. \cite{ErdosPower}, pp.223,224).

\begin{proposition}\label{corrinj0}
Let $\E$  be a nest in $\B(X)$ and let $\Phi, \Theta$ be  admissible support functions on $\E$ such that $\mathcal M(\Phi)=\mathcal M(\Theta)$, where $\mathcal M(\Phi)$ and  $\mathcal M(\Theta)$ are the bimodules associated with, respectively,  $\Phi$ and  $\Theta$,  defined in \eqref{01}. Then, for all $E\in\E\backslash\{\{0\}\}$,  $\Phi(E)=\Theta(E)$.
\end{proposition}

\begin{proof}
Let $E\in\E$ and  
suppose firstly that $E=E_-\neq\{0\}$. Let $F$ be any subspace of the nest such that $F\subset E$. Let $x\in\Phi(F)$ and $f\in F^{\perp}\backslash E^{\perp}$. We show next that  $f\otimes x \in \mathcal M(\Phi)=\mathcal M(\Theta)$. 

Let $N\in \E$. If  $N\subset F$, then 
$$
f\otimes x(N)=f(N)x=\{0\}\subseteq \Phi(N).
$$
If, on the other hand, $N$ is such that $F\subseteq N$, then 
$$
f\otimes x(N)\subseteq\text{span}\{x\}\subseteq\Phi(F)\subseteq\Phi(N),
$$
 from which follows that 
  $f\otimes x\in\mathcal M(\Phi)$. 
  Since $f\otimes x\in\mathcal M(\Theta)$, we have that $f\otimes x(E)\subseteq \Theta(E)$. 
  Moreover, since $f$ was chosen not to lie in $E^{\perp}$,
   we have that $f(E)\neq\{0\}$ and, consequently, $x\in\Theta(E).$ 
   Notice that this holds for every $x\in\Phi(F)$. Hence $\Phi(F)\subseteq\Theta(E)$, for every $F\subset E.$ Now, using the left continuity of $\Phi$, we have that
$$
\Phi(E)=\Phi(E_-)=\vee\{\Phi(F):F\subset E\}\subseteq\Theta(E).
$$
Similarly, it can be shown   that $\Theta(E)\subseteq \Phi(E)$, and we have finally that  $\Phi(E)=\Theta(E).$

Suppose now that $E_-\subset E$. Let $x\in \Phi(E)$   and let $f\in E_-^{\perp}\backslash E^{\perp}$. Then $f\otimes x\in\mathcal M(\Phi)$. 
 In fact, if $N\subset E$, then $N\subseteq E_-$. Hence,  
 $$
 f\otimes x(N)=f(N)x\subseteq f(E_-)x=\{0\}\subseteq\Phi(N).
 $$
  If, on the other hand, $E\subseteq N$, then 
  $$
  f\otimes x(N)\subseteq\text{span}\{x\}\subseteq\Phi(E)\subseteq\Phi(N).
  $$
   Since $f\otimes x\in\mathcal M(\Theta)$, we have that 
   $$
   f\otimes x(E)=f(E)x\subseteq\Theta(E).
   $$
 Since $f\not\in E^{\perp}$, it follows that $x\in\Theta(E)$. Keeping in mind that $x\in\Phi(E)$ was arbitrarily chosen, we have  $\Phi(E)\subseteq\Theta(E)$. 
  It can be shown similarly that 
 $\Theta(E)\subseteq \Phi(E)$. Hence $\Phi(E)=\Theta(E)$, as required.
\end{proof}

\begin{remark}\label{rem1}Observe that the image of $\{0\}$ under a support function $\Phi$ is of no consequence regarding  the bimodule associated with $\Phi$. Hence, we see that there exists a bijection between reflexive bimodules and admissible support functions fixing $\{0\}$. 
\end{remark}

\begin{example}\label{exa1} The bijection referred to in Remark \ref{rem1}  no longer exists, if we consider general support functions.

Let $X=L^2([0,1])$ and let $\E$ be the continuous nest $\E=\{N_t:t\in [0,1]\}$ with 
$$N_t=\{f\in L^2([0,1])\colon f(s)=0\text{ a.e. on } [t,1]\}.
$$ 
Consider the support functions $\Phi$ and $\Psi$ on $\E$ given by 
$$
\Phi(N_t)=\begin{cases}
N_t &   0\le t<\frac 12 \\
 X & \frac 12\le t\le 1
\end{cases}\,,
\qquad \qquad
\Psi(N_t)=\begin{cases}
N_t &   0\le t\le\frac 12 \\
 X & \frac 12< t\le 1
\end{cases}\,.
$$
The support function $\Psi$ is admissible whilst $\Phi$ is not. However, we do have $\mathcal M(\Phi)=\mathcal M(\Psi)$. Indeed, the inclusion $\mathcal M (\Phi)\supseteq \mathcal M(\Psi)$ follows from $\Psi\le\Phi$. For the other inclusion, notice that, given $T\in\mathcal M(\Phi)$, we have 
$$TN_t\subseteq\Phi(N_t)=\Psi(N_t),
$$ for $t\neq \frac 12$,  and also 
$$TN_{\frac 12}=\vee_{t<\frac 12} TN_t\subseteq \vee_{t<\frac 12} \Phi(N_t)=\vee_{t<\frac 12} N_t=N_{\frac 12}=\Psi(N_{\frac 12}).$$
Hence $T\in\mathcal M(\Psi).$
\end{example}

\begin{proposition}\label{corrinj1}
 Let $\E$  be a nest in $\B(X)$ and let $\Phi$ be a support function on $\E$ such that  $\Phi(\{0\})=\{0\}$. Then the following hold.
 
 \begin{enumerate}
 \item [(i)] $\mathcal M(\Phi)=\mathcal M(\Phi_-)$.   
 \item [(ii)] $
 \Phi_{\mathcal M(\Phi)}=\Phi_-
 $.
\item [(iii)] If $\Phi$ is an admissible support function on $\E$,  then $\Phi_{\mathcal M(\Phi)}=\Phi.$
\end{enumerate}
\end{proposition}

\begin{remark}\label{rem3}
It is clear from definition \eqref{01} and this proposition that, given an admissible support function $\Phi$ fixing $\{0\}$, the  maximal (weakly closed) bimodule having $\Phi$ as its   support function is $\mathcal M(\Phi)$.
\end{remark}

\begin{proof}
(i) We begin by showing that $\mathcal M(\Phi)=\mathcal M(\Phi_-)$. Since $\Phi_-\le\Phi$ we have the inclusion $\mathcal M(\Phi)\supseteq\mathcal M(\Phi_-)$.  

To prove that the reverse inclusion holds,  let $T\in\mathcal M(\Phi)$ and let $E\in\E$. 
We wish to show that $TE\subseteq\Phi_-(E)$. If $E=\{0\}$ or $E_-\subset E$, then $\Phi_-(E)=\Phi(E)$. Hence, we have   that $TE\subseteq\Phi_-(E)$ as a direct consequence of $T$ lying in $\mathcal M(\Phi)$.

 Suppose now that $\{0\}\neq E_-=E.$ Notice that, for all $F$ contained in $ E$, we have $TF\subseteq \Phi(F)$. Hence, 
$$
\vee_{F\subset E}TF\subseteq\vee_{F\subset E}\Phi(F)=\Phi_-(E).
$$

By the continuity of $T$, 
$$\vee_{F\subset E}TF=T\left(\vee_{F\subset E}F\right)=TE_-=TE,
$$
as required.

(ii) Notice that 
$$
\Phi_-(\{0\})=\Phi(\{0\})=\{0\}=\Phi_{\mathcal M(\Phi)}(\{0\}).
$$
 By Proposition \ref{corrinj0} and (i) of this proposition,  we have 
 $\mathcal M(\Phi_-)=\mathcal M(\Phi)=\mathcal M(\Phi_{\mathcal M(\Phi)})
 $. Hence, since $\Phi_-$ and $\Phi_{\mathcal M(\Phi)}$ are both admissible, it follows from Remark \ref{rem1} that $\Phi_{\mathcal M(\Phi)}=\Phi_-.$  

Assertion (iii) is an immediate consequence of (ii).
\end{proof}

A reflexive operator space is weakly closed but it is not necessarily true that a weakly closed operator space is reflexive.  However, we shall see   that   $\T(\E)$-bimodules  are weakly closed if and only if they are reflexive. The remainder of this section is devoted precisely to showing that these two classes of $\T(\E)$-bimodules coincide. A key tool to obtain this characterisation are the rank-$1$ and, more generally, the finite rank operators in a given bimodule.    The next three results  concern  these operators.

\begin{proposition}\label{p_aux}
 Let $\E$ be a nest, let $\T(\E)$ be the corresponding nest algebra and let $\J$ be a 
$\T(\E)$-bimodule.  Let $L,N\in\E$ be such that there exists $T\in\J$ with $TN\not\subseteq L_-$. Then, for every $f\in N_-^{\perp}$ and $x\in L$, the operator $f\otimes x$ lies in $\J$.
\end{proposition}
\begin{proof}
Let $f\in N_-^{\perp}$ and $x\in L$. Since $TN\not\subseteq L_-$, there exists $y\in N$ such that $Ty\not\in L_-$. Hence, there exists $g\in L_-^{\perp}$ such that $g(Ty)\neq 0$. 

By Lemma \ref{rankonenestalg},  we have that     $f\otimes y$ and $g\otimes x$ lie in $\mathcal T(\E)$, since $x\in L\subseteq (L_-)_+$. Hence
$$
(g\otimes x)T(f\otimes y)=(g\otimes x)(f\otimes Ty)=g(Ty)(f\otimes x)
$$
lies in $\J$. Since $g(Ty)\neq 0$, it follows that $f\otimes x\in\J$.   
\end{proof}

The following lemma characterising the rank-$1$ operators in $\mathcal M(\Phi)$ is essentialy [\cite{ErdosPower}, Lemma 1.1].

\begin{lemma}\label{bimodrankcar}
Let $\mathcal E$ be a nest in $\B(X)$ and let $\Phi$ be a support function on $\mathcal E$. Let $w\in X$ and let $f$ be a functional in $ X^*$. Then  
 the rank-$1$ operator $f\otimes w$ lies in  $\mathcal M(\Phi)$ if and only if there exists $E\in\mathcal E$ such that $f\in E^{\perp}$ and $w\in \wedge_{E\subset F}\Phi(F)$.
\end{lemma}

Although the proof of the next theorem be similar to that of [\cite{LiLu}, Theorem 3.1], we include it here for the reader's convenience.

\begin{theorem}\label{finiteranksum}
Let $\mathcal E$ be a nest in $\B(X)$, let $\Phi$ be a support function on $\mathcal E$ and let $T$ be a rank-$n$ operator in $\mathcal M(\Phi)$. Then $T$ can be written as a sum of $n$ rank-1 operators in $\mathcal M(\Phi)$.
\end{theorem}

\begin{proof}
Let 
$
T=\sum_{i=1}^n{f_i\otimes x_i} 
$
 be a finite rank operator in $\mathcal M(\Phi)$, where $x_i\in X, f_i\in X^*$, for all $i=1,...,n$. The proof will use induction on the rank $n$ of $T$. 

If $n=0,1$ the result holds trivially. Assume that $n>1$ and that every finite rank operator in $\mathcal M(\Phi)$ of rank $m\le n-1$ is the sum of $m$ rank one operators in $\mathcal M(\Phi)$. 

Let $W$ be the finite dimensional space  $W=T(X)$ and let 
 $L=\wedge\{N\in\E:N\cap W\neq \{0\}\}.
 $
We have that $L\cap W\neq \{0\}$ (see Lemma \ref{finitesubspacenotzero} and its proof). 

Let  $x\in L\cap W$ be  a non-zero vector,  let $g\in X^*$ be such that $g(x)=1$ and define $f\in X^*$  by  $f(y)=g(Ty)$ for all $y\in X$. 
 We  prove now that $f\otimes x\in\mathcal M(\Phi)$. 
 
 Let $N\in\E$. If $\Phi(N)\subset L$, then $TN\subseteq W\cap \Phi(N)=\{0\}$. Hence,
  $$f\otimes x(N)=g(TN)x=\{0\}\subseteq\Phi(N).
  $$
   If on the other hand $\Phi(N)\supseteq L$, then $L\cap W\subseteq \Phi(N)\cap W$.
   Hence $x\in\Phi(N)\cap W$ and we have that 
   $$f\otimes x(N)\subseteq\text{span}\{x\}\subseteq \Phi(N),
   $$
   yielding  that $f\otimes x\in\mathcal M(\Phi)$. 
   
   We  show next that $T-f\otimes x$ has rank $\le n-1$, which, by the induction hypothesis, is enough to finish the proof.  Since $x\in W=T(X)$, there exist $\alpha_1, \dots, \alpha_n\in \complexos$ such that 
$
x=\sum_{i=1}^n\alpha_ix_i
$.
 Hence
$$T-f\otimes x=\sum_{i=1}^n{(f_i-\alpha_if)\otimes x_i}.
$$
 We need only to show that $\{f_i-\alpha_if:i=1,...,n\}$ is a linearly dependent set. 
 
 By the definition of $f$,  
$$
0=T^*g\otimes x-f\otimes x=(g\otimes x)(T-f\otimes x)=\left(\sum_{i=1}^ng(x_i)(f_i-\alpha_if)\right)\otimes x
$$
from which follows that $\sum_{i=1}^ng(x_i)(f_i-\alpha_if)=0$. But, since 
$$
1=g(x)=\sum_{i=1}^n\alpha_ig(x_i),
$$
 there exists $i\in\{1,...,n\}$ such that  $g(x_i)\neq 0$. Hence $\{f_i-\alpha_if:i=1,...,n\}$ is a linearly dependent set, as required.
\end{proof}

Given a nest $\E$ and its   corresponding nest algebra $\T(\E)$,  let $\J$ be a $\T(\E)$-bimodule. We denote by $\J_0$ the $\T(\E)$-bimodule contained in  $\J$ consisting of  its finite rank operators.

\begin{proposition}\label{samesupport}
Let $\Phi$ be an admissible support function on $\E$ and let $\J$ be a $\T(\E)$-bimodule with associated support function $\Phi_\J$. Then, the support functions $\Phi_{\J_0}$ and $\Phi_{\overline{\J_0}}$ 
of the $\T(\E)$-bimodules $\J_0$ 
and $\overline{\J_0}$, respectively, are such that
$$\Phi_{\J_0}=\Phi_{\overline{\J_0}}=\Phi_\J.
$$	
\end{proposition}
\begin{proof}
 Since the inclusion $[\mathcal JN]\supseteq[\mathcal J_0N]$ and the  equality $[\mathcal J_0N]=[\overline{\mathcal J_0}N]$ are obvious, it suffices to show that,  for all $N\in\mathcal E$,
 $$
 [\mathcal JN]\subseteq[\mathcal J_0N].
 $$ 
 By Theorem \ref{rankstrong}, there exists  a net $\{R_{\alpha}\}$ of finite rank operators in $\mathcal T(\E)$ converging to $I$ in the strong operator topology. It follows that, for any $J\in\mathcal J$,   $\{JR_{\alpha}\}$ is a net of finite rank operators in $\mathcal J_0$ strongly converging to $J$. Hence, given $x\in N$,  we have $Jx = \lim JR_{\alpha}x$ and, consequently, $Jx\in[\mathcal J_0x]$. Since $x$ and $J$ are arbitrary, 
  $[\mathcal JN]\subseteq[\mathcal J_0N]$, which concludes the proof.
\end{proof}

\begin{proposition}\label{p_prop1}
Given a nest $\E$ and its   corresponding nest algebra $\T(\E)$,  let $\J$ be a $\T(\E)$-bimodule. Then $\J\subseteq \overline{\J_0}^{\text{SOT}}\subseteq \overline{\J_0}^{\text{WOT}}$.
\end{proposition}

\begin{proof}
 Let $J\in\J$. By Theorem \ref{rankstrong} there exists a net $\{R_{\alpha}\}$ of finite rank operators in $\mathcal T(\mathcal E)$ converging to the identity in the strong operator topology. Then, clearly, $\{R_{\alpha}J\}$ is a net of finite rank operators in $\J$ converging to $J$ in the strong operator topology proving the first inclusion above.	
  The remaining inclusion is obvious.
\end{proof}

The following corollary is an immediate consequence of this proposition.

\begin {corollary}\label{weakspan}
Let $\J$ be a weakly closed $\T(\E)$-bimodule.
Then $\J=\overline{\J_0}^{\text{WOT}}$. 
\end{corollary}

\begin{lemma}\label{rankoneban}
Given a nest $\E$ and its   corresponding nest algebra $\T(\E)$,  let $\J$ be a weakly closed $\T(\E)$-bimodule and let $\Phi_{\mathcal J}$ be the support function associated with $\J$. 
 Then the $\T(\E)$-bimodules $\mathcal J$ and $\mathcal M(\Phi_{\mathcal J})$ contain the same finite rank operators.
\end{lemma}
\begin{proof}
Since $\mathcal J\subseteq \mathcal M(\Phi_{\mathcal J})$, we only need to show that the set of finite rank operators in  $\mathcal M(\Phi_{\mathcal J})$  is contained in $\J_0$. 
Since, by Proposition \ref{finiteranksum}, any  finite rank operator $T$ in $\mathcal M(\Phi_\J)$ can be split into a sum of rank-$1$ operators in $\mathcal M(\Phi_\J)$, it is enough to show that the set of rank-$1$ operators in  $\mathcal M(\Phi_{\mathcal J})$  is contained in $\J_0$. This will conclude the proof. 

Let $f\otimes w$ be a rank-$1$ operator in $M(\Phi_{\mathcal J})$. By Proposition \ref{bimodrankcar}, there exists $E\in\mathcal E$ such that $f\in E^{\perp}$ and $w\in \wedge_{E\subset F}\Phi_\J(F)$. 

By Lemma \ref{perpspan}, there exists a net $(N_\alpha)$ in $\E$ and  a net  $(f_{\alpha})$ in $X^*$   such that, for each $\alpha$,  $f_\alpha\in N_\alpha^{\perp}$ and  $(N_{\alpha})_+\supset E$,  and $(f_{\alpha})$ converges to $f$ in weak*- topology. 

For each $\alpha$, let $w_\alpha\in (N_{\alpha})_+\backslash E$. Notice that, by the reflexivity of $\E$, $[\mathcal T(\E)w_\alpha ]\in\E$. Since  $w_\alpha\in [\mathcal T(\E)w_\alpha ]$,  we must have the strict inclusion $E\subset [\mathcal T(\mathcal E)w_\alpha]$. It follows that 
$$
\wedge_{E\subset F}\Phi_\J(F)\subseteq \Phi_{\mathcal J}([\mathcal T(\mathcal E) w_\alpha])=[\mathcal J[\mathcal T(\mathcal E) w_\alpha]]\subseteq [\mathcal J w_\alpha].
$$
 Hence, for each $\alpha$, we can find $J_\alpha\in\mathcal J$ such that $(J_\alpha w_\alpha)$ converges to $w$ in the norm topology.

Since $f_\alpha\in N_\alpha^{\perp}$ and $ w_\alpha \in (N_{\alpha})_+$, it follows from  Lemma \ref{rankonenestalg} that $ f_\alpha\otimes w_\alpha\in\mathcal T(\mathcal E)$. Hence
$
J_\alpha(f_\alpha\otimes w_\alpha)=f_\alpha\otimes J_\alpha w_\alpha 
$ 
lies in $\J$.

Given that $(f_\alpha)$ converges to $f$ in the weak*-topology and $(J_{\alpha}w_\alpha)$ converges  to $w$ in the norm toplogy, we have that $(f_\alpha\otimes J_\alpha w_\alpha)$  converges weakly to $f\otimes w$. Since $\mathcal J$ is weakly closed we have  that $f\otimes w\in\mathcal J$.
\end{proof}

We are now ready to prove that the weakly closed bimodules of a Banach space nest algebra coincide with its reflexive bimodules.

\begin{theorem}\label{closedcar}
Let $\E$ be a nest, let $\T(\E)$ be the   corresponding nest algebra  and let $\J$ be a weakly closed $\T(\E)$-bimodule.
 Then $\J=\mathcal M(\supp)$ and $\J$ is reflexive.
\end{theorem}

\begin{proof}
The bimodules $\J$ and $\mathcal M(\Phi_\J)$ are both weakly closed $\T(\E)$-bimodules. 
Since, by Lemma \ref{rankoneban}, $\J$ and $\mathcal M(\Phi_\J)$ have the same set of finite rank operators, it follows from  Corollary \ref{weakspan} that $\J$ and $\mathcal M(\Phi_{\J})$ coincide.
\end{proof}

\section{Essential support functions}\label{s_essbim}

  We define below the key notions of this section, those of essential support function and    admissible support function  pair, and show how  $\T(\E)$-bimodules determine and are determined by  essential support functions and  admissible support function  pairs. It should be noted that, although our definitions of essential support function and admissible support function  pair are inspired  by those  in \cite{DDH} for the Hilbert space setting, they are not quite a generalisation. 
  To mention one of the differences, notwithstanding   each essential support function $\Psi$ determining a largest $\T(\E)$-bimodule whose essential support function is $\Psi$, the latter is not necessarily closed, unlike its Hilbert space counterpart  (see Example \ref{exa}).
   It is also the case that  each essential support function   determines a smallest norm closed $\T(\E)$-bimodule whose essential support funtion is $\Psi$ (see Theorem \ref{p_bimodess}).
  
  Given an admissible support function  pair $(\Phi, \Psi)$, we find the largest $\T(\E)$-bimodule whose   support function pair is precisely $(\Phi, \Psi)$, if the nest $\E$ has the so-called $p$-property. As to finding the minimal bimodule corresponding to $(\Phi, \Psi)$, we make some headway for nests having the $p_\infty$-property to be defined below. In this setting, we are able to construct a bimodule which we conjecture to be minimal amongst the norm closed $\T(\E)$-bimodules having $(\Phi, \Psi)$ as its support function pair.  Although each of these results be obtained restricting the nests to a class, it is also the case that the $p$-property is enjoyed by a vast class of nests whereas the $p_\infty$-property corresponds to a restriction already existing in  Hilbert space nest algebras (see \cite[Section 2]{DDH}).\\

Given a nest $\E$ in $X$, let $\E_f$ denote the set of elements $N\in\E$ such that $0<\dim(N/N_-)<\infty$, and let $\E_{\infty}=\E\backslash \E_f$. Here, for subspaces $N$ and $M$ of $X$, we adopt the notation 
$$M/N=\{m+N:m\in M\}.
$$ 

An \emph{essential support function} on $\E$ is a support function $\Psi$ such that, for all $N, N_1,N_2\in\E$ with  $N_1\subseteq N_2$, 
\begin{equation}\label{B}
\Psi(N)\in\E_f\Rightarrow \Psi(N)=\Psi(N)_+;
\end{equation}
\begin{equation}\label{A}
	 \dim\left(N_2/N_1\right)<\infty \Rightarrow \Psi(N_2)=\Psi(N_1).
	 \end{equation}

A \emph{support function pair} consists of a pair $(\Phi,\Psi)$ of support functions on $\E$, where $\Phi$ is an admissible support function with $\Phi(\{0\})=\{0\}$ and  $\Psi$ is an essential support function with $\Psi\le \Phi$. The pair $(\Phi,\Psi)$ is said to be an \emph{admissible support function pair} on $\E$ if, for all $N\in \E$,
\begin{equation}\label{B'}
\Psi(N)\in\E_f \Rightarrow \Psi(N)\subset \Phi(N).
\end{equation}

We aim now  to establish a correspondence between  bimodules and essential support functions. We begin with the definitions of the essential support function and  support function pair of a bimodule.

Given a $\T(\E)$-bimodule $\J$, the \emph{essential support function of $\J$} is the function $\suppe\colon\E\to\E$  defined, for all $N\in\E$,  by 
\begin{equation}\label{suppedef}
\suppe(N)=\wedge\{L\in\E:\dim\left(TN/L)<\infty\quad\forall T\in\J\right\}.
\end{equation}
The \emph{support function pair of $\J$} is defined as the pair $(\supp,\suppe)$ of support functions.

\begin{remark}\label{rem2} 
Observe that, if $L\subset \suppe(N)$, then there exists $T\in\J$ such that $\dim(TN/L)=\infty$. It is also the case  that, if $L\supset\suppe(N)$, then $\dim(TN/L)<\infty$ for every $T\in\J$. 

The essential support functions defined in \cite{DDH} for Hilbert space nest algebras  localise where the non-compact operators are supported  in a given bimodule. Here, however, this is no longer so, although in some sense our essential  support functions be still {\sl sensitive} to the non-compact operators.
\end{remark}

\begin{lemma}\label{cor1}
Let $\E$ be a nest, let $\T(\E)$ be the corresponding nest algebra and  let $\J$ be a $\T(\E)$-bimodule with essential support function $\suppe$. Let  $L,N\in\E$ be  such that $L_-\subset \suppe(N)$. Then, for every $f\in N_-^{\perp}$ and $x\in L$,  the operator $f\otimes x$ lies in $\J$ . 
\end{lemma}

\begin{proof}
Since  
$$
\suppe(N)=\wedge\{L\in\E:\dim(TN/L)<\infty \quad\forall T\in\J\},
$$
 the inclusion $L_-\subset \suppe(N)$ implies that there exists $T\in\J$ such that $\dim(TN/L_-)=\infty$. In particular, $\dim(TN/L_-)\neq 0$ and, therefore,  $TN\not\subseteq L_-$. The result now follows from Proposition \ref{p_aux}.
\end{proof}

\begin{proposition}\label{proofess}
Let $\E$ be a nest, let $\T(\E)$ be the corresponding nest algebra and let $\J$ be a $\T(\E)$-bimodule. Then 
$(\supp,\suppe)$ is an admissible  support function pair on $\E$.
\end{proposition}

\begin{proof}
It is easily seen that  $\suppe$  is a support function on $\E$. 
 We  show firstly that $\suppe$ satisfies property \eqref{B}. 
 Let $N$ be a subspace in $\E$ and suppose that $0<\dim(\suppe(N)/\suppe(N)_{-})<\infty$.

To prove that $\suppe(N)=\suppe(N)_{+}$, it is enough to show that 
\begin{equation}\label{phie}
\dim(TN/\suppe(N))=\infty, \qquad   \mbox{for some $T\in\mathcal J$}.
\end{equation} 
 To see this, suppose that \eqref{phie} holds and let  $L\in \E$.   Then $\dim(TN/L)<\infty$ for every $T\in\J$ if and only if $\suppe(N)\subset L$ and, consequently, 
$$
\suppe(N)=\wedge\{L\in\E:\dim\left(TN/L)<\infty\quad\forall T\in\J\right\}=\wedge\{L\in\E:\suppe(N)\subset L\}=\suppe(N)_+.
$$

We show now that \eqref{phie}  holds. Suppose  that $\dim(TN/\suppe(N))<\infty$ for all operators $T$ in the bimodule $\J$. Let $\{e_1+\suppe(N),...,e_k+\suppe(N)\}$ be a basis of the subspace $TN/\suppe(N)$ and observe that we have $TN=\suppe(N)+\text{span}\{e_1,...,e_k\}$. 
It follows that 
$$
\dim(TN/\suppe(N)_-)\le\dim(\suppe(N)/\suppe(N)_-)+\dim(\text{span}\{e_1,...,e_k\}/\suppe(N)_-)<\infty.
$$
 But this contradicts the definition of $\suppe(N)$, since $\suppe(N)_-\subset \suppe(N)$.

We show now that \eqref{A} holds. Let $N_1,N_2\in \E $ be such that $N_1\subset N_2$, $\dim(N_2/N_1)<\infty$. By \eqref{suppedef}, it is enough to show that, whenever $L\in\E, T\in\mathcal J$ are such that  and $\dim(TN_1/L)<\infty$, then $\dim(TN_2/L)<\infty$. But, if $\{b_1+N_1,...,b_r+N_1\}$ is a basis of $N_2/N_1$, then $N_2=N_1+\text{span}\{b_1,...,b_r\}$ and, therefore, 
$$
\dim(TN_2/L)\le\dim(TN_1/L)+\dim(T(\text{span}\{b_1,...,b_r\})/L)<\infty.
$$

We prove now that the conditions of \eqref{B'} apply. It is obvious that   $\supp(\{0\})=\{0\}$ and by Proposition \ref{p_phiJ}, $\supp$ is an admissible support function. 

 To see that  $\suppe\leq\supp$, we must show that, for any $N\in \E$, we have $\suppe(N)\subseteq\supp(N)$. Observe that, for all $N\in \E$,  $\supp(N)=[\mathcal JN]\supseteq TN$. Hence, we have  
 \begin{equation}\label{dimzero}
 \dim(TN/\supp(N))=0<\infty, \qquad \mbox{for all $N\in \E$},
 \end{equation}
  from which follows that $\suppe(N)\subseteq\supp(N)$. 

Finally, we prove that \eqref{B'} holds. Let $N\in \E$ be such that $\suppe(N)\in\E_f$. We have already shown in the previous paragraph that   $\suppe(N)\subseteq \supp(N)$. However, if  $\suppe(N)=\supp(N)$, then, by \eqref{phie}, we would have $\dim(TN/\supp(N))=\infty$, for some $T\in\J$, which cannot be, as seen in \eqref{dimzero}. 
\end{proof}

Given an essential support function $\Psi$ on a nest $\E$ in $\B(X)$, let $\mathcal M^e(\Psi)$ and $\mathcal M^0(\Psi)$ be  the subspaces of $\mathcal B(X)$ defined, respectively, by 
\begin{equation}\label{02}
\mathcal M^e(\Psi)=\{T\in\mathcal B(X):\dim\left(TN/L)<\infty\quad\forall N,L\in\E\text{ with } L_+\supset \Psi(N)\right\},
\end{equation}

\begin{equation}\label{10}
\mathcal M^0(\Psi)=\overline{\sum_{\substack{L,N\in\E \\ L_-\subset \Psi(N)}}{\text{span}\{f\otimes x:f\in N_-^{\perp},x\in L\}}}.
\end{equation}
(cf. \cite{DDH}).

\begin{theorem}\label{p_bimodess}
Let $\E$ be a nest, let $\T(\E)$ be the corresponding nest algebra and let $\Psi$ be an essential support function on $\E$. Then the following hold.
\begin{enumerate}
\item[(i)] $\mathcal M^e(\Psi)$ is a  $\T(\E)$-bimodule containing every $\T(\E)$-bimodule $\J$ whose  essential support function $\suppe$ coincides with $\Psi$.

\item [(ii)]  $\mathcal M^0(\Psi)$ is a $\T(\E)$-bimodule  contained in every norm closed $\T(\E)$-bimodule $\J$ for which  $\suppe=\Psi$. 
\end{enumerate}
\end{theorem}

\begin{proof}
(i) It is clear that $\mathcal M^e(\Psi)$ is a linear subspace of $\mathcal B(X)$. 

To see that $\mathcal M^e(\Psi)$ is indeed a bimodule let $A,B\in\mathcal T(\E)$, let $T\in\mathcal M^e(\Psi)$ and let $N,L\in\E\text{ be such that } L_+\supset \Psi(N)$. We shall show  that $\dim(ATBN/L)<\infty$. 

 Firstly, observe that $\dim(
ATBN/L)\le \dim(ATN/L)$, since $BN\subseteq N$. Hence,  it suffices to prove  that $\dim(ATN/L)<\infty$.  

 On the contrary, suppose that $\dim(ATN/L)=\infty$ and let $(x_n)$ be  a sequence $N$ such that $\{ATx_n+L:n\in\mathbb N\}$ is a linearly independent subset of $ATN/L$. Hence,  since $AL\subseteq L$, we  also have that$\{Tx_n+L:n\in\mathbb N\}$ is also a linearly independent subset of  $TN/L$. But, since $T\in\mathcal M^e(\Psi)$, this contradicts the fact that $\dim(TN/L)<\infty$.
 
 The remaining assertion in (i) follows immediately from \eqref{02}.

(ii) Clearly $\mathcal M^0(\Psi)$ is a linear subspace of $\mathcal B(X)$. Let $L,N\in\E$ be such that $L_-\subset\Psi(N)$ and let $f\in N_-^\perp$ and $x\in L$. 
To see that $\mathcal M^0(\Psi)$ is a bimodule, it is enough to show that $A(f\otimes x)B\in \mathcal M^0(\Psi)$, for every $A,B\in\mathcal T(\E).$ 
 
 The operator $A(f\otimes x)B$ coincides with  $g\otimes (Ax)$, where $g\in X^*$ is given by $g(y)=f(By)$. Since $AL\subseteq L$, we have that $Ax\in L$. 
 Moreover, since $BN_-\subseteq N_-$, we have that $f\in N_-^\perp$ implies that $g\in N_-^\perp$. But by Lemma \ref{cor1}. this shows that $g\otimes (Ax)\in \mathcal M^0(\Psi)$.
 
 By Lemma \ref{cor1}, $\mathcal M^0(\Psi)$ is contained in every norm closed bimodule $\J$ such that  $\suppe=\Psi$.  
\end{proof}

\begin{example}\label{exa}
 It is worth noticing that $\mathcal M^e(\Psi)$ need not be a closed subspace. Let $X=L^2([0,1])$ and  consider the continuous nest   $\mathcal E=\{N_t:t\in[0,1]\}$, with 
 $$
 N_t=\{f\in L^2([0,1])\colon f(s)=0\text{ a.e. on }[t,1]\}.
 $$
  Let  $\Psi$ be the essential support function on $\E$ such that $\Psi(N_t)=\{0\}$, if $0\le t\le \frac 12$,  and $\Psi(N_t)=X$, if $\frac 12 <t\le 1$.
 
  Having fixed a basis $\{e_k:k\in\mathbb N\}$  of $N_{\frac 12}$, define, for $n\in \mathbb N$,  the finite rank operators 
  $T_n$ to be the projections onto $\sspan\{e_1,...,e_n\}$ and let   $T=\lim_{n\to\infty} T_n$ be  the projection onto $N_{\frac 12}$. 
  
 The  sequence $(T_n)$ is contained in $\mathcal M^e(\Psi)$, since it consists of finite rank operators. However, that is not the case of its limit $T$. For example,  we have that 
 $$
 (N_{\frac 14})_+=N_{\frac 14}\supset \{0\}=\Psi(N_{\frac 12})
 $$
  but 
  $
 \dim (TN_{\frac 12}/N_{\frac 14})=\infty.
 $
 This shows that $\mathcal M^e(\Psi)$ is not closed, unlike 
 its    counterpart  defined only in the Hilbert space setting (cf. \cite{DDH}, p. 63). As we see, the present definition of essential support function is not quite a generalisation of that in \cite{DDH} yielding  notwithstanding similar results. 
\end{example}

In our ongoing investigation of essential support functions, admissible support function pairs and their associated bimodules, we shall restrict ourselves to the case where the nest has the so-called $p$-property, to be defined below.  However,  we shall see that this is not a severe restriction. As will be clear upon its definition, this property holds in a vast class of Banach space nest algebras being, in fact, automatic in the case of separable Hilbert spaces (see  \cite{BCW, MO}).

A nest $\E$ is said to have  the \emph{$p$-property} if the following holds. 
\begin{enumerate} 
\item[(i)]For $N\in\E$ with $N=N_-$, there exists a strictly increasing sequence $(N_k)$ in $\E$ such that $\vee N_k=N$.
\item[(ii)] For $N\in\E$  with $N=N_+$, there exists a strictly decreasing sequence $(N_k)$ in $\E$ such that $\wedge N_k=N$.
\end{enumerate}

\begin{lemma}\label{maxess}
Let $\E$ be a nest having the $p$-property. If $\Psi$ is an essential support function on $\E$, then $$\Phi^e_{\mathcal M^e(\Psi)}=\Psi.$$
\end{lemma}

\begin{remark}\label{rem4}
It is worth pointing out that, by this lemma and  Theorem \ref{p_bimodess},  the set $\mathcal M^e(\Psi)$ is the maximal $\T(\E)$-bimodule with essential support function $\Psi$, whenever the nest $\E$ has the $p$-property.
\end{remark}

\begin{proof}

Fix  $N_0\in \E$ and, to simplify the notation, let $\phie=\Phi^e_{\mathcal M^e(\Psi)}$. We start by observing that $\phie(N_0)\subseteq \Psi(N_0)$. Indeed, by \eqref{02}, for every $L\in\E$ with $L_+\supset \Psi(N_0)$ and every $T\in\mathcal M^e(\Psi)$, we have $\dim(TN_0/L)<\infty$. Hence, by \eqref{suppedef}, we have  $\phie(N_0)\subseteq L$. It follows that 
$$\Phi^e(N_0)\subseteq\wedge_{L_+\supset \Psi(N_0)}L=\Psi(N_0).
$$ 
We show next that $\phie(N_0)\supseteq \Psi(N_0)$ which will end the proof. For that, it suffices to find an operator $T$ such that
\begin{itemize}
\item[(A)] $T\in\mathcal M^e(\Psi)$;
\item[(B)] $\dim(TN_0/L)=\infty, \text{ for every } L\subset \Psi(N_0).$
\end{itemize} 
In order to construct such an operator, we shall analyse below the three possible situations: 
  $N_0 ={N_0}_-$ (Case 1),  $\dim(N_0/{N_0}_-)=\infty$ (Case 2) and $0<\dim(N_0/{N_0}_-)<\infty$ (Case 3).

\noindent
{\sl Case 1} ($N_0 ={N_0}_-$): 
 Let $(N_k)$ be a strictly increasing sequence  in $\E$ such that $\vee N_k=N_0$, and choose a sequence  $(f_k)$ such that, for all $k\in \naturais$, we have  $f_k\in N_k^{\perp}\backslash N_{k+1}^{\perp}$.

{\sl Case 1.1} ($\Psi(N_0)=\Psi(N_0)_-$): 
 Let $(L_k)$ in $\E$ be a strictly increasing sequence such that $\vee L_k=\Psi(N_0)$. Fix a sequence $(e_k)$ in $L_{k+1}\backslash L_k$  such that, for all $k\in \naturais$, we have $\left\|f_k\right\|<\frac{\left\|e_k\right\|}{k^2}$. Let $T$ be the operator defined by the absolutely convergent series $T=\sum_{k= 1}^\infty{f_k\otimes e_k}$.
 We prove next that $T\in\mathcal M^e(\Psi)$. 
 
 Let $L,N\in\E$ be such that $L_+\supset \Psi(N)$. 
 If $N\subseteq N_1$, then $TN=\{0\}$ and, hence, $\dim(TN/L)=0<\infty$. 
 
 If $N_1\subset N\subset N_0$, then there exists an integer $k\ge 1$ such that $N_k\subset N\subseteq N_{k+1}$. Hence, for every $l\ge k+1$, we have $f_l(N)=\{0\}$. It follows that $TN\subset\span \{e_1,...,e_k\}$ and, consequently,  $\dim(TN/L)<\infty$.
 
Suppose now that $N\supset N_0$. By the monotonicity of $\Psi$, we have 
$
\Psi(N)\supseteq\Psi(N_0).
$
 Since 
 $$
 TN\subseteq\overline{\text{span}}\{e_k\colon k\in \naturais\}\subseteq \Psi(N_0),
 $$
  it  follows that $\dim(TN/L)\le\dim (TN/\Psi(N_0))=0$. Hence $T\in\mathcal M^e(\Psi)$. 
   
It remains to show that, for $L\in\E$ with $L\subset \Psi(N_0)$, we have $\dim(TN_0/L)=\infty$. 

Let $m$ be an integer such that $L\subset L_m$. The set $\{e_k+L\colon k\ge m\}$ is a linearly independent set in $TN_0/L$. 
Suppose that, on the contrary,   there exist scalars $a_1,...,a_n$, not all equal to zero, such that
$$a_1(e_{k_1}+L)+...+a_n(e_{k_n}+L)=0+L,
$$
where $m\le k_1<... <k_n$. We have then that $a_1e_{k_1}+...+a_ne_{k_n}$  lies in $L$.
Assuming without loss of generality that $a_n\neq 0$, it follows that 
$$
e_{k_n}=a_n^{-1}((a_1e_{k_1}+...+a_ne_{k_n})-(a_1e_{k_1}+...+a_{n-1}e_{k_{n-1}}))
$$
lies in $L_{k_n}$, since  $e_{k_1},...,e_{k_{n-1}}\in L_{k_n}$ and  $L\subset L_{k_n}$.
 But this contradicts the fact that  $e_{k_n}\in L_{k_n+1}\backslash L_{k_n},$  concluding the proof of this case.

{\sl Case 1.2} ($\dim(\Psi(N_0)/\Psi(N_0)_-)=\infty$):
 Fix a subset $\{e_k\colon k\in \naturais\}$ of $\Psi(N_0)$ such that 
 $\{e_k+\Psi(N_0)_-\colon k\in \naturais\}
 $ is a linearly independent set in $\Psi(N_0)/\Psi(N_0)_-$. Consider the operator $T=\sum_{k=1}^\infty{f_k\otimes e_k}$ defined as above. We can show similarly that $T\in\mathcal M^e(\Psi)$.

To see  that $\dim (TN_0/L)=\infty$, for every $L\subset \Psi(N_0)$, notice that 
$$
TN_0=\overline{\text{span}}\{e_k\colon k\in \naturais\}
$$
 and that 
 $\{e_k+L\colon k\in \naturais \}
 $
  is a linearly independent set in $TN_0/L$. 
  Indeed, if  $a_1,...,a_n$ are scalars such that 
  $$a_1(e_{k_1}+L)+...+a_n(e_{k_n}+L)=0+L,
  $$
  then $a_1e_{k_1}+...+a_ne_{k_n} \in \Psi(N_0)_-$. That is,
   $$
   a_1(e_{k_1}+\Psi(N_0)_-)+...+a_n(e_{k_n}+\Psi(N_0)_-)=0+\Psi(N_0)_-,
   $$
    yielding that the scalars  $a_1,...,a_n$ must all coincide with zero.

{\sl Case 1.3} ($0< \dim(\Psi(N_0)/\Psi(N_0)_-)<\infty$):
 Observe that in this case property \eqref{B} implies that $\Psi(N_0)=\Psi(N_0)_+$. 
 
 For any positive integer $k$, let $(L_k)$ be  a strictly decreasing sequence in $\E$ with   
 $
 \wedge L_k=\Psi(N_0).
 $
 and fix a sequence $(e_k)$ where, for each $k$, $e_k\in L_k\backslash L_{k+1}$. 
  We show next  that  the operator $T=\sum_{k\ge 1}{f_k\otimes e_k}$ lies in $\mathcal M^e(\Psi)$.

Let $L,N\in\E$ be such that $L_+\supset \Psi(N)$. If $N\subset N_0$, it can be seen similarly to Case 1.1 that $\dim(TN/L)<\infty$. 

If $N_0\subseteq N$, then $L_+\supset \Psi(N)$ implies $L\supset \Psi(N_0)$. Indeed, the last inclusion must be strict, since $L=\Psi(N_0)$ would imply 
$$\Psi(N_0)\subseteq \Psi(N)\subset L_+=\Psi(N_0)_+=\Psi(N_0),
$$ yielding a contradiction.
If $L_1\subseteq L$, then 
$$
\dim(TN/L)\le\dim(L_1/L)=0<\infty.
$$
 If $\Psi(N_0)\subset L\subset L_1$, then there exists a positive integer $k$ such that $L_{k+1}\subset L\subseteq L_k$. 
 It follows that, for all $l\in \naturais$ with $l\ge k+1$, the vector $e_l$ lies in  $L$. Hence
  $$
  \dim(TN/L)\le\dim(\span\{e_1,...,e_k\}/L)<\infty,
  $$ 
  as required.

Finally, we show that, for every $L\subset \Psi(N_0)$, we have $\dim(TN_0/L)=\infty$. We check that $\{e_k+L:k\in \naturais\}$ is a linearly independent set in $TN_0/L$. 

Suppose on the contrary that for integers $1\le k_1<... <k_n$ and  scalars $a_1,...,a_n$ not all zero, we had
$$a_1(e_{k_1}+L)+...+a_n(e_{k_n}+L)=0+L.
$$
If we assume without loss of generality that $a_1\neq 0$, then 
\begin{equation}\label{07}
e_{k_1}=a_1^{-1}((a_1e_{k_1}+...+a_ne_{k_n})-(a_2e_{k_2}+...+a_{n}e_{k_n}))
\end{equation}
Since   $e_{k_2},...,e_{k_{n}}\in L_{k_2}$,  
 $L\subset L_{k_2}$ and, as seen above, $a_1e_{k_1}+...+a_ne_{k_n}\in L$, it follows from \eqref{07} that $e_{k_1}$ lies in  $L_{k_2}$, contradicting the assumption that $e_{k_1}\in L_{k_1}\backslash L_{k_2}.$ \\

\noindent
{\sl Case 2} ($\dim(N_0/(N_0)_-)=\infty$): In this case, it suffices to choose a linearly independent set $\{f_k\colon k\in \naturais \}$ in $(N_{0})_-^{\perp}\backslash N_0^{\perp}$ and use a reasoning similar to Case 1.\\

\noindent
{\sl Case 3} ($0<\dim(N_0/({N_0})_-)<\infty$): Define 
 $$N_1=\wedge\{M\in\E:M\subset N_0\text{ and }\dim(N_0/M)<\infty\}.
 $$
If $\dim(N_0/N_1)<\infty$, then $N_1\in\E_{\infty}$. Hence, similarly to the cases above, we have $\phie(N_1)=\Psi(N_1)$. Moreover,   property \eqref{A}   yields $\phie(N_0)=\phie(N_1)$ and $\Psi(N_0)=\Psi(N_1)$. But then $\phie(N_0)=\Psi(N_0)$, as required.

If, on the other hand,  $\dim(N_0/N_1)=\infty$, let $\{f_k\colon k\in \naturais\}$  be a linearly independent set in $(N_{1})^{\perp}\backslash N_0^{\perp}$, and consider the two possible situations 
(i)  $\Psi(N_0)\in\E_{\infty}$ and 
 (ii) $\Psi(N_0)\in\E_{f}$.\\

\noindent
(i) $\Psi(N_0)\in\E_{\infty}$.

(i.1) If $\dim(\Psi(N_0)/\Psi(N_0)_-)=\infty$, choose a subset $\{e_k\colon k\in \naturais\}$ of $\Psi(N_0)$ such that 
$\{e_k+\Psi(N_0)_-\colon k\in \naturais\}
$ is linearly independent in $\Psi(N_0)/\Psi(N_0)_-$.  Consider again the operator 
$T=\sum_{k\ge 1}{f_k\otimes e_k}
$   which we show next to lie in $M^e(\Psi)$.

Let $L,N\in\E$ be such that $L_+\supset \Psi(N)$. 
It is clear that, if $N\subseteq N_1$,  then $TN=\{0\}$ and $\dim (TN/L)=0$. 

Suppose now that $N_1\subset N\subseteq N_0$. Since $\dim(N_0/N)<\infty$, by \eqref{A}, we have $\Psi(N)=\Psi(N_0)$. 
 Hence 
 $$
 L\supseteq \Psi(N)=\Psi(N_0)
 $$ 
 and it follows that 
 $$
 \dim(TN/L)\le\dim(TN/\Psi(N_0))=0.
 $$
  Finally, if $N\supset N_0$,  a reasoning similar to that of Case 1.1 yields $\dim(TN/L)=0$. We have shown that $T\in\mathcal M^e(\Psi)$.

To finish this case just notice that $T$ satisfies property (B) for the same reasons as those in Case 1.2.

(i.2) If  on the other hand, we have $\Psi(N_0)\in\E_{\infty}$ and $\Psi(N_0)=\Psi(N_0)_-$, 
let  $(L_k)$ be a strictly increasing sequence    in $\E$ such that $\vee L_k=\Psi(N_0)$.  Consider again the operator $T=\sum_{k= 1}^\infty{f_k\otimes e_k}$ where, for each $k$, $e_k\in L_{k+1}\backslash L_{k}$. 

A reasoning similar to that in (i.1) shows that $T\in\mathcal M^e(\Psi)$.  Similarly to the final considerations in Case 1.1, 
we obtain that $T$ satisfies property (B).\\

\noindent
(ii) $\Psi(N_0)\in\E_{f}$.

Property \eqref{B} implies that $\Psi(N_0)=\Psi(N_0)_+$. Let  $(L_k)$ be a strictly decreasing sequence    in $\E$ such that $\wedge L_k=\Psi(N_0)$, and let $T=\sum_{k= 1}^\infty{f_k\otimes e_k}$  where, for each $k$, $e_k\in L_{k}\backslash L_{k+1}$.


As in (i.1), 
 we have $T\in\mathcal M^e(\Psi)$. Moreover, a reasoning similar to that in Case 1.3  yields that $T$ satisfies property (B).
 
%

The proof is complete.
\end{proof}

\begin{theorem}\label{t_maxpair}
Let $\E$ be a nest having the  $p$-property, let $\T(\E)$ be the corresponding nest algebra and let $(\Phi,\Psi)$ be an admissible support function pair  on $\E$. Then 
\begin{equation}\label{05}
\mathcal M(\Phi,\Psi)=\mathcal M(\Phi)\cap\mathcal M^e(\Psi)
\end{equation}
 is the largest $\mathcal T(\E)$-bimodule whose support function pair $(\supm,\supme)$ coincides with $(\Phi,\Psi)$.
\end{theorem}

\begin{proof}
Clearly $\mathcal M(\Phi,\Psi)$ is a bimodule and it contains every bimodule with support $\Phi$ and essential support $\Psi$. Let $\mathcal \B(X)_0$ the set of finite rank operators on $X$. 
Since $\mathcal \B(X)_0\subseteq \mathcal M^e(\Psi)$, we have 
$$
\mathcal M(\Phi)_0=\mathcal M(\Phi)\cap \mathcal \B(X)_0\subseteq\mathcal M(\Phi,\Psi)\subseteq \mathcal M(\Phi).
$$
 By Proposition \ref{samesupport}, the bimodule $\mathcal M(\Phi)_0$ has the same support function as its closure $\mathcal M(\Phi)$.
 Hence  the support $\supm$  coincides with $\Phi$.

 It remains to show that 
  $\Phi^e_{\mathcal M(\Phi,\Psi)}=\Psi$.
 To simplify the notation, let $\phie=\Phi^e_{\mathcal M(\Phi,\Psi)}$. 
 
 Notice that, since $\mathcal M(\Phi,\Psi)\subseteq\mathcal M^e(\Psi)$, we have that $\phie\leq\Psi$. We see next that  the reverse inequality holds.
 
For a fixed $N_0\in \E$, we need to show that  $\phie(N_0)\subseteq\Psi(N_0)$.  For that it is wnough to find an operator $T$ such that
\begin{itemize}
\item[(A)] $T\in\mathcal M^e(\Psi)$;
\item[(B)] $\dim(TN_0/L)=\infty, \text{ for every } L\subset \Psi(N_0);$
\item[(C)] $T\in\mathcal M(\Phi)$.
\end{itemize}
 We shall consider the operator $T$ to be as in the proof of Lemma \ref{maxess} and shall adopt here also a  notation and case labelling  as in that proof, 
 unless otherwise specified. Referring to that proof,  it is clear that (A) and (B) will hold similarly.  Hence we shall only show that (C) holds. That is, we fix $N\in\E$ and prove that  $TN\subseteq\Phi(N)$. 

Observe that, given $L\subset \Psi(N_0)$, it is always possible to find $M\in \E$ with $M\subset N_0$  and $L\subset \Phi(M)$.  This is a consequence of the fact that 
$\Psi(N_0)\subseteq\Phi(N_0)
$ and  of the left continuity of $\Phi$.


\noindent
{\sl Case 1} ($N_0 ={N_0}_-$): It is clear that, if  $N\subseteq N_1$, then $TN=\{0\}\subseteq\Phi(N)$.
 Suppose then that $N_1\subset N$. 

{\sl Case 1.1} ($\Psi(N_0) ={\Psi(N_0)}_-$): Choose the sequence $(N_k)$ to be such  that $L_{k}\subset \Phi(N_{k-1})$, whenever $k\ge 2$. 
It follows that, if $N_1\subset N\subset N_0$, then there exists an integer $k\ge 1$ such that $N_k\subset N\subseteq N_{k+1}$, and therefore 
$$
TN\subseteq \text{span}\{e_1,...,e_k\}\subseteq L_{k+1}\subset \Phi(N_k)\subseteq\Phi(N).
$$
If on the other hand $N_0\subseteq N$, then 
 $$
 TN\subseteq \Psi(N_0) \subseteq\Phi(N_0)\subseteq\Phi(N).
 $$

 {\sl Case 1.2} ($\dim (\Psi(N_0)/{\Psi(N_0)}_-)=\infty$): We just need to choose $N_1$ such that ${\Psi(N_0)}_-\subset \Phi(N_1)$ and (i) follows similarly.

 {\sl Case 1.3} ($0<\dim (\Psi(N_0)/{\Psi(N_0)}_-)<\infty$): In this case we know that $\Psi(N_0)=\Psi(N_0)_+\subset \Phi(N_0)$. By the left continuity of $\Phi$, we can choose the subspace $N_1\subset N_0$ to be such that $\Psi(N_0)\subset \Phi(N_1)\subset\Phi(N_0)$. Hence, in this case we let  the linear space $L_1$ to be $\Phi(N_1)$. It follows that 
 $$TN=\overline{\text{span}}\{e_1,e_2,...\}\subseteq L_1=\Phi(N_1)\subseteq\Phi(N).
 $$

 \quad\\
\noindent
 {\sl Case 2} ($\dim(N_0/{N_0}_-)=\infty$): In this case, we can use a reasoning similar to the case above to prove (i).

 \quad\\
\noindent
 {\sl Case 3} ($0<\dim(N_0/{N_0}_-)<\infty$):  
Recall that, here, 
$$
N_1=\wedge\{M\in\E:M\subset N_0\text{ and }\dim(N_0/M)<\infty\}.
$$
If $\dim(N_0/N_1)<\infty$, then $N_1\in\E_{\infty}$. Hence, similarly to the cases above, we have $\phie(N_1)=\Psi(N_1)$. Moreover,   property \eqref{A}   yields $\phie(N_0)=\phie(N_1)$ and $\Psi(N_0)=\Psi(N_1)$. But then $\phie(N_0)=\Psi(N_0)$, as required.

Hence we suppose that $\dim (N_0/N_1)=\infty$.
If $\Psi(N_0)\in\E_{\infty}$, condition (C) is immediately verified. In fact,
if $N\subseteq N_1$, then $TN=\{0\}\subseteq\Phi(N)$.   If we have $N_1\subset N\subset N_0$, 
$$
TN\subset \Psi(N_0)=\Psi(N)\subseteq\Phi(N)
$$ (here we used property \eqref{A}). 
Finally, if $N_0\subseteq N$, then 
$$
TN\subset \Psi(N_0)\subseteq\Psi(N)\subseteq\Phi(N).
$$

Now suppose that $\Psi(N_0)\in\E_f$, that is, $0<\dim(\Psi(N_0)/\Psi(N_0)_-)<\infty$.   
  Here the choice of the sequences $(f_k)$ and  $(L_k)$   must be slightly different. Notice that, by the definition of $N_1$ and the assumption $\dim(N_0/N_1)=\infty$, we have that $N_1=(N_1)_+$. Hence choose a strictly decreasing sequence $(M_k)$ in $\E$ 
   with  $\wedge M_k=N_1$.

 Notice that for every integer $k\ge1$, we have $\dim(N_0/M_k)<\infty$. By properties \eqref{A} and \eqref{B'} of the pair $(\Phi,\Psi)$, we have 
 $$
\Psi(N_0)=\Psi(M_k)\subset\Phi(M_k).
 $$
 We may therefore assume, by means of a subsequence, that $L_{k}\subseteq \Phi(M_{k+1})$, for every integer $k\ge 1$.

Let $f_k\in M_{k+1}^{\perp}\backslash M_{k}^{\perp}$ and $e_k\in L_{k}\backslash L_{k+1}$, for all integer $k\ge 1$. Conditions (A) and (B) can be shown to hold Just like in the proof of Lemma \ref{maxess}. We show now that (C) holds.

If $N\subseteq N_1$, then $TN=\{0\}\subseteq\Phi(N)$. If $N_1\subset N\subseteq M_1$, then there exists an integer $k\ge 1$ such that $M_{k+1}\subset N\subseteq M_{k}$ and, therefore, 
$$
TN\subseteq\overline{\text{span}}\{e_j:j\ge k\}\subseteq L_{k}\subseteq\Phi(M_{k+1})\subseteq\Phi(N).
$$
 Finally, if $M_1\subset N$, then $TN\subseteq L_1\subseteq\Phi(M_2)\subseteq \Phi(N),$
 ending the proof.
\end{proof}

Having characterised the largest bimodule with a given support function pair $(\Phi,\Psi)$, we now turn to the problem of finding the smallest bimodule whose  support function pair is  $(\Phi,\Psi)$. 
 To tackle this problem, we assume the nest to satisfy the so-called $p_\infty$-property. The same problem was addressed in \cite{DDH} in similarly restrictive circumstances. In fact, in \cite{DDH}, only infinite multiplicity nests were considered (see \cite[Section 2]{DDH}).

Although we do not give a complete answer to this problem, we  make some headway, nevertheless. We construct a  (possibly not  closed)  bimodule with support function pair $(\Phi,\Psi)$ (see \eqref{08}), which we conjecture to be contained in every closed bimodule with support function pair ($\Phi,\Psi$).

 We say that a nest $\E$ has \emph{$p_{\infty}$-property}   if, for each $N\in\E$ with $N_-\subset N$, we have   $$\dim(N/N_-)=\infty
 $$
(cf. with the restriction made to the nests in \cite[Section 2]{DDH}).

\begin{remark}\label{rem7}
Observe  that for a nest having the $p_{\infty}$-property, conditions  \eqref{B}, \eqref{A} and \eqref{B'}  hold trivially.
Then it is also the case  that every support function on $\E$ is an essential support function and, for a pair $(\Phi,\Psi)$ of support functions to be admissible it is only required that $\Phi$ be admissible, $\Psi\le\Phi$ and $\Phi(\{0\})=\{0\}$.
\end{remark}

\begin{theorem}\label{miness}
Let $\E$ be a nest having the  $p_{\infty}$-property and let $\Psi$ be a(n essential) support function on $\E$ such that  $\Psi(\{0\})=\{0\}$. Then the  support function $\Phi_{\mathcal M^0(\Psi)}$ and the  essential support function $\Phi^e_{\mathcal M^0(\Psi)}$ of the bimodule $\mathcal M^0(\Psi)$ coincide with  $\Psi_-$, i.e.,
$$
\Phi_{\mathcal M^0(\Psi)}=\Psi_-=\Phi^e_{\mathcal M^0(\Psi)}.
$$
\end{theorem}
\begin{proof}
Denote by $\Phi^e$ the essential support function of $\mathcal M^0(\Psi)$. Observe  that 
$$\Phi^e(\{0\})=\{0\}=\Psi(\{0\})=\Psi_-(\{0\}).
$$
Let $N_0$ be a subspace in $\E$. We start by showing that $\Phi^e(N_0)\subseteq \Psi_-(N_0)$. For that just observe that

$$\begin{aligned}
\Phi^e(N_0)&\subseteq\Phi_{\mathcal M^0(\Psi)}(N_0)=[\mathcal M^0(\Psi)(N_0)] \\
&=\overline{\text{span}}\{(f\otimes x)(N_0):L_-\subset\Psi(N),x\in L, f\in N_-^{\perp}\} \\
&=\overline{\text{span}}\{(f\otimes x)(N_0):L_-\subset\Psi(N),x\in L, f\in N_-^{\perp}\backslash N_0^{\perp}\}\\
&=\overline{\text{span}}\{x:L_-\subset\Psi(N),x\in L, N_0^{\perp}\subset N_-^{\perp}\}\\
&=\overline{\text{span}}\{x:L_-\subset\Psi(N),x\in L, N_-\subset N_0\}\\
&=\vee_{\substack{L_-\subset\Psi(N) \\ N_-\subset N_0}}L=\vee_{N_-\subset N_0}\Psi(N)=\Psi_-(N_0).
\end{aligned}
$$ 

Notice that, the equalities above, show that $\Phi_{\mathcal M^0(\Psi)}=\Psi_-$.

It only remains to show that $\Phi^e(N_0)\supseteq \Psi_-(N_0)$. We divide the proof in two cases: Case (i) ($N_{0-}\subset N_0$) and Case (ii) $N_{0-}= N_0.$
\vspace{0.2cm}

{\sl Case (i)}: $N_{0-}\subset N_0.$

Since $\E$ has the  $p_{\infty}$-property, we have that $\dim(N_0/N_{0-})=\infty$. 
Let $\{x_n+N_{0-}:n\in\mathbb N\}$ be a linearly independent set in $N_0/N_{0-}$. 
Since no non-trivial linear combination of the elements of $\{x_n+N_{0-}:n\in\mathbb N\}$ lies  in $N_{0-}$,  we have that, for all $n\in\mathbb N$, the vector $x_n$ does not lie in $N_{0-}+\text{span}\{x_1,...,x_{n-1}\}$.
Hence, for all $n\in\mathbb N$,  we can choose  $f_n\in N_{0-}^{\perp}$  such that $f_n(x_1)=\cdots=f_n(x_{n-1})=0$ and $f_n(x_n)\neq 0$.  

Let $L\in\E$ be such that $L_-\subset \Psi(N_0)$ and let $M\subset L$. Since $\E$ has the $p_\infty$-property,   $\dim(L/M)=\infty$, and we can choose a linearly independent subset $\{l_n+M:n\in\mathbb N\}$ of $L/M$. 

By \eqref{10}, we have that $f_n\otimes l_n\in\mathcal M^0(\Psi)$, for every $n\in\mathbb N$. 
Let $T$ be  the operator
$$T=\sum_{n=1}^{\infty}f_n\otimes l_n,
$$
where $\left\|f_n\right\|\le \left\|l_n\right\|/n^2$,  rendering the defining series of $T$  absolutely convergent. Since $\mathcal M^0(\Psi)$ is norm closed, it follows that $T\in\mathcal M^0(\Psi)$.

We see next that  $\dim(TN_0/M)=\infty$. We shall show that the set $\{Tx_n+M:n\in\mathbb N\}$ is linearly independent  in $TN_0/M$. 

Suppose on the contrary that there exist positive integers $n_1<\cdots< n_k$ and  scalars $a_1,...,a_k$, with $a_k\neq 0$, such that $$T(a_1x_{n_1}+\cdots+a_kx_{n_k})=
 \sum_{n=1}^{n_k}f_n(a_1x_{n_1}+\cdots+a_kx_{n_k})l_n
$$
lies in $M$ and assume without loss of generality that $a_k\neq 0$.
 Hence, we must have 
$$f_{n_k}(a_1x_{n_1}+\cdots+a_kx_{n_k})=0
$$ from which follows that $a_kf_{n_k}(x_{n_k})=0$. 
But $f_{n_k}(x_{n_k})\neq 0$ and, consequently,  $a_k=0$, yielding  a contradiction.

Since $\dim(TN_0/M)=\infty$, we have that $\Phi^e(N_0)\supseteq M_+$. In view of this holding   for every $M\subset L$, we can conclude that
$$
\Phi^e(N_0)\supseteq\vee_{\substack{M\subset L \\L_-\subset\Psi(N_0)}}M_+=\vee_{L_-\subset\Psi(N_0)}L=\Psi_-(N_0).$$

{\sl Case (i)}: $N_{0-}= N_0.$

Let $L\in\E$ be such that 
$$L_-\subset \Psi_-(N_0)=\vee_{M\subset N_0}\Psi(M).
$$
 Hence,  there exists $M\subset N_0$ such that $L_-\subset \Psi(M)$. 

Observing that  $\dim(N_0/M_-)=\infty$, we can choose a linearly independent set $\{x_n+M_-:n\in\mathbb N\}$   in $N_0/M_-$,  
 $f_n\in M_-^{\perp}$ such that 
$$f_n(x_1)=\cdots f_n(x_{n-1})=0
$$ with  $f_n(x_n)\neq 0$.

Notice that for  $L'\subset L$, we have $\dim(L/L')=\infty$ and, therefore, we can choose a linearly independent set  $\{l_n+L':n\in\mathbb N\}$  in $L/L'$. 

As in the previous case we can choose the sequence $(f_n)$ to be such that the operator 
$T=\sum_{n=1}^{\infty}f_n\otimes l_n$ lies in $\mathcal M^0(\Psi)
$. 

Similarly to the case above, we have that $\dim(TN_0/L')=\infty$ and that, therefore,  $\Phi^e(N_0)\supset L'_+.$ Hence
$$
\Phi^e(N_0)\supseteq\bigvee_{\substack{L'\subset L \\L_-\subset\Psi_-(N_0)}}L'_+=\bigvee_{L_-\subset\Psi_-(N_0)}L=\Psi_-(N_0), 
$$ 
as required.
\end{proof}

Bearing in mind that adding finite rank operator to a bimodule does not change its essential support function, we make the following definition.

Let $(\Phi,\Psi)$ be  an admissible  support function pair on a nest $\E$. Define 

\begin{equation}\label{08}
\mathcal M^0(\Phi,\Psi)=\mathcal M^0(\Psi)+\mathcal M(\Phi)_0.
\end{equation}

\begin{proposition}\label{p_m^0}
Let $\E$ be a nest with the $p_{\infty}$-property, let $\T(\E)$ be the corresponding nest algebra  and let $(\Phi,\Psi)$ be an admissible support function pair on $\E$. Then $\mathcal M^0(\Phi,\Psi)$ is a $\T(\E)$-bimodule whose support function pair is $(\Phi,\Psi_-)$. 
\end{proposition}

\begin{proof}
As observed above, adding finite rank operators does not change the essential support function. In consequence, the essential support function of $\mathcal M^0(\Phi,\Psi)$ coincides with that of $\mathcal M^0(\Psi)$, which has been shown  to be $\Psi_-$.  

In  the proof of Proposition \ref{miness}, we saw that the support function of $\mathcal M^0(\Psi)$ is $\Psi_-$. Since $\Psi_-\le \Psi\le\Phi$, it follows that $\mathcal M^0(\Psi)\subseteq\mathcal M(\Phi)$. 
Hence
 $$
 \mathcal M(\Phi)_0\subseteq \mathcal M^0(\Phi,\Psi)\subseteq\mathcal M(\Phi).
 $$

But, by Proposition \ref{samesupport}, 
  $$\Phi_{\mathcal M(\Phi)_0}=\Phi_{\mathcal M(\Phi)}=\Phi
  $$
   from which follows that $\mathcal M^0(\Phi,\Psi)$ also has support function $\Phi$. 
\end{proof}

  \begin{corollary}\label{c_m^0}
  Let $\E$ be a nest with the $p_{\infty}$-property, let $\T(\E)$ be the corresponding nest algebra  and let $(\Phi,\Psi)$ be an admissible support function pair on $\E$. If 
   $\Psi$ is left continuous, then $\mathcal M^0(\Phi,\Psi)$ is a $\T(\E)$-bimodule whose support function pair is $(\Phi,\Psi)$. 
\end{corollary}

\begin{proof}
This is an immediate consequence of Proposition \ref{p_m^0}, since in this case $\Psi=\Psi_-$
\end{proof}

As seen in Proposition \ref{p_bimodess}, $\mathcal M^0(\Psi)$ is contained in every norm closed $\T(\E)$-bimodule  with essential support function $\Psi$. On the other hand, by Lemma  \ref{rankoneban}, $\mathcal M_0(\Phi)$  is contained in every norm closed $\T(\E)$-bimodule with support function $\Phi$. In view of this, we propose the following.

\begin{conjecture}\label{co_final}
Let $\E$ be a nest, let $\T(\E)$ be the corresponding nest algebra, let $\J$ be a norm closed $\T(\E)$-bimodule   and let 
$(\Phi,\Psi)$ be the support function pair of $\J$. Then $\J$ contains $\mathcal M^0(\Phi,\Psi)$.
\end{conjecture}

\end{document}